\documentclass{elsarticle}

\usepackage{amsmath}
\usepackage{amssymb}
\usepackage{latexsym}
\usepackage{hyperref}

\newcommand{\EQ}{\begin{equation}}
\newcommand{\EN}{\end{equation}}

\newtheorem{theorem}{Theorem}[section]
\newtheorem{definition}[theorem]{Definition}
\newtheorem{corollary}[theorem]{Corollary}

\newtheorem{lemma}[theorem]{Lemma}

\newdefinition{remark}[theorem]{Remark}

\newproof{proof}{Proof}

\newcommand{\rank}{\mbox{\rm rank}}

\newcommand{\ba}{{\bf a}}
\newcommand{\bb}{{\bf b}}

\newcommand{\bh}{{\bf h}}
\newcommand{\bc}{{\bf c}}

\newcommand{\by}{{\bf y}}
\newcommand{\bx}{{\bf x}}

\newcommand{\bv}{{\bf v}}

\newcommand{\bo}{{\bf 0}}

\newcommand{\F}{\mathbb{F}}

\newcommand{\al}{\alpha}

\newcommand{\wt}{\operatorname{wt}}
\newcommand{\Aut}{\operatorname{Aut}}

\title{Completely regular codes with different parameters and the same distance-regular coset graphs}

\author[jr]{J. Rif\`{a}\corref{cor1}}
\ead{josep.rifa@uab.cat}

\author[vz]{V.A. Zinoviev}
\ead{zinov@iitp.ru}

\cortext[cor1]{Corresponding author}
\address[jr]{Department of Information and Communications Engineering,
Universitat Aut\`{o}noma de Barcelona, 08193-Bellaterra, Spain}

\address[vz]{Kharkevich Institute for Information Transmission Problems,
Russian Academy of Sciences, Bol'shoi Karetnyi per. 19, GSP-4,
Moscow, 127994, Russia}

\begin{document}

\begin{abstract}
A known Kronecker construction of completely regular codes has been investigated taking
different alphabets in the component codes. This approach is also connected
with lifting constructions of completely regular codes. We obtain
several classes of completely regular codes with different parameters,
but identical intersection array. Given a prime power $q$ and any two
natural numbers $a,b$, we construct completely transitive codes over
different fields with covering radius $\rho=\min\{a,b\}$ and  identical
intersection array, specifically, one code over $\F_{q^r}$ for each
divisor $r$ of $a$ or $b$. As a
corollary, for any prime power $q$, we show that distance
regular bilinear forms graphs can be obtained as coset graphs from
several completely regular codes with different parameters. Under
the same conditions, an explicit construction of an infinite family
of $q$-ary uniformly packed codes (in the wide sense) with
covering radius $\rho$, which are not completely regular, is also
given.
\end{abstract}

\maketitle

\section{Introduction}

Let $\F_q$ be a finite field of the order $q$ and
$\F_q^*=\F_q \setminus \{0\}$. A
$q$-ary linear code $C$ of length $n$ is a $k$-dimensional subspace of
$\F_q^n$. Given any
vector $\bv \in \F_q^n$, its {\em distance to the code
$C$} is $d(\bv,C)=\min_{\bx \in C}\{ d(\bv, \bx)\}$, the \textit{minimum distance} of the code is $d=\min_{\bv \in C} \{d(\bv, C\backslash \{\bv\})\}$
and the {\em covering radius} of the code $C$ is
$\rho=\max_{\bv \in \F_q^n} \{d(\bv, C)\}$. We say that $C$ is a $[n,k,d;\rho]_q$-code. Let $~D=C+\bx~$ be a
{\em coset} of  $C$, where $+$ means the component-wise
addition in $\F_q$. The {\em weight} $\wt(D)$
of $D$ is the minimum weight of the codewords of $D$.

For a given $q$-ary code $C$ with covering radius $\rho=\rho(C)$
define
\[
C(i)~=~\{\bx \in \F_q^n:\;d(\bx,C)=i\},\;\;i=1,2,...,\rho.
\]
Say that two vectors $\bx$ and $\by$ are {\em neighbors} if
$d(\bx,\by)=1$. For two vectors $\bx=(x_1,\ldots, x_n)$ and
$\by=(y_1, \ldots, y_n)$ over $\F_q$ denote by $\bx\by$ their
inner product over $\F_q$, i.e.\[ \bx\by = x_1y_1 + \ldots +
x_ny_n.
\]

\begin{definition}\cite{Neum}\label{de:1.1} A $q$ary code $C$ is completely regular, if
for all $l\geq 0$ every vector $x \in C(l)$ has the same number
$c_l$ of neighbors in $C(l-1)$ and the same number $b_l$ of
neighbors in $C(l+1)$. Define $a_l = (q-1)n-b_l-c_l$ and
$c_0=b_\rho=0$. Denote by $(b_0, \ldots, b_{\rho-1}; c_1,\ldots,
c_{\rho})$ the intersection array of $C$.
\end{definition}

The equivalent definition of completely regular codes is due to
Delsarte \cite{Del1}.

\begin{definition}\label{de:1.2}\cite{Del1}
A $q$-ary code $C$ with covering radius $\rho$ is called
{\em completely regular} if the weight distribution of any translated $D$
of $C$ of weight $i$, $i=0,1,...,\rho$ is uniquely defined by the
minimum weight of $D$, i.e. by the number $i=\wt(D)$.
\end{definition}

Let $M$ be a monomial matrix, i.e. a matrix with exactly one
nonzero entry in each row and column. If $q$ is prime, then
$\Aut(C)$ consist of all monomial ($n\times n$)-matrices $M$ over
$\F_q$ such that $\bc M \in C$ for all $\bc \in C$. If $q$ is a
power of a prime number, then $\Aut(C)$ also contains any field
automorphism of $\F_q$ which preserves $C$. The group $\Aut(C)$
acts on the set of cosets of $C$ in the following way: for all
$\sigma\in \Aut(C)$ and for every vector $\bv \in \F_q^n$ we have
$(\bv + C)^\sigma = \bv^\sigma + C$.

\begin{definition}\label{de:1.3}\cite{Giudici,Sole}
Let $C$ be a linear code over $\F_q$ with covering radius $\rho$.
Then $C$ is completely transitive if $\Aut(C)$ has $\rho +1$
orbits when acts on the cosets of $C$.
\end{definition}

Since two cosets in the same orbit should have the same weight
distribution, it is clear, that any completely transitive code is
completely regular.

\begin{definition}\label{de:1.4}\cite{Bas1}
Let $C$ be a $q$-ary code of length $n$ and let $\rho$ be its
covering radius. We say that $C$ is {\em uniformly packed} in the
wide sense, i.e. in the sense of \cite{Bas1}, if there exist
rational numbers $~\al_0,~\ldots,~\al_{\rho}~$ such that for any
$\bv~\in~\F_q^n~$ \EQ\label{eq:2.1}
\sum_{k=0}^{\rho}\al_k\,f_k(\bv)~=~1 , \EN where $f_k(\bv)$ is the
number of codewords at distance $k$ from $\bv$.
\end{definition}

Completely regular and completely transitive codes are classical subjects
in algebraic coding theory, which are closely connected with graph theory,
combinatorial designs and algebraic combinatorics. Existence, construction and
enumeration of all such codes are open hard problems (see \cite{Bro1,Neum,Gill,Kool}
and references there).

In a recent paper \cite{Rif2} we described an explicit
construction, based on the Kronecker product of parity check matrices,
which provides, for any natural number $\rho$ and for any prime
power $q$, an infinite family of $q$-ary linear completely regular
codes with covering radius $\rho$. In \cite{Rif3} we
presented another class of $q$-ary linear completely regular
codes with the same property, based on lifting of perfect codes.
Here we extend the Kronecker product construction
to the case when component codes have
different alphabets and connect the resulting completely regular
codes with codes obtained by lifting  $q$-ary perfect codes.
This gives several different infinite classes of completely
regular codes with different parameters and with
identical intersection arrays.

\section{Preliminary results}

For a $q$-ary $[n,k,d;\rho]_q$ code $C$ let $s = s(C)$ be its {\em
outer distance}, i.e. the number of different nonzero weights of
codewords in the dual code $C^{\perp}$.

\begin{lemma}\label{lem:2.1}
Let $C$ be a code with covering radius $\rho$ and external
distance $s$. Then,
\begin{enumerate}[i)]
\item \cite{Del1}~$\rho(C) \leq s(C)$.
\item \cite{Bas2}~The code $C$ is
uniformly packed in the wide sense if, and only if, $\rho=s$.
\item \cite{Bro1}~If $C$ is completely regular then it is
uniformly packed in the wide sense.
\end{enumerate}
\end{lemma}

\begin{lemma}\cite{Neum}\label{lem:2.2}
Let $C$ be a linear completely regular code with intersection
array $(b_0, \ldots, b_{\rho-1}; c_1,\ldots, c_{\rho})$, and let
$\mu_i$ be the number of cosets of $C$ of weight $i$. Then
\[
\mu_ib_i = \mu_{i+1}c_{i+1}.
\]
\end{lemma}

\begin{definition}\label{de:2.1}
For two matrices $A=[a_{r,s}]$ and $B =[b_{i,j}]$ over $\F_q$
define a new matrix $H$ which is the Kronecker product $H = A
\otimes B$, where $H$ is obtained by changing any element
$a_{r,s}$ in $A$ by the matrix $a_{r,s} B$.
\end{definition}

Consider the matrix $H = A \otimes B$ and let $C$, $C_A$ and $C_B$
be the codes over $\F_q$ which have, respectively, $H$, $A$ and
$B$ as parity check matrices. Assume that $A$ and $B$ have size
$m_a \times n_a$ and $m_b \times n_b$, respectively. For
$r\in\{1,\ldots, m_a\}$ and $s\in \{1,\ldots, m_b\}$ the rows in
$H$ look as
$$(a_{r,1}b_{s,1}, \ldots, a_{r,1}b_{s,n_b}, a_{r,2}b_{s,1}, \ldots,
a_{r,2}b_{s,n_b}, \ldots, a_{r,n_a}b_{s,1}, \ldots,
a_{r,n_a}b_{s,n_b}).$$ Arrange these rows taking blocks of $n_b$
coordinates as columns such that the codewords in code $C$ are
presented as matrices $\left[\bc\right]$ of size $n_b \times n_a$:
\EQ\label{eq:2.0}
\left[\bc\right]= \left[
\begin{array}{ccc}
c_{1,1} & \ldots & c_{1,n_a}\\
c_{2,1} & \ldots & c_{2,n_a}\\
\vdots &\vdots &\vdots \\
c_{n_b,1} & \ldots & c_{n_b,n_a}
\end{array}
\right] = \left[
\begin{array}{ccc}
&\bc_1 & \\
&\bc_2 & \\
&\vdots &\\
&\bc_{n_b} &
\end{array}
\right] = \left[
\bc^{(1)}\,\bc^{(2)}\,\ldots\,\bc^{(n_a)}
\right],
\EN
where $c_{i,j} = a_{r,j} b_{s,i}$,\;$\bc_r$ is the
$r$th row vector of the matrix $C$ and $\bc^{(\ell)}$ is its
$\ell$th column.

The following result was obtained in \cite{Rif2}.

\begin{theorem}\label{theo:2.1}
Let $C(H)$ be the $[n,k,d;\rho]_q$ code with parity check matrix
$H=A\otimes B$ where $A$ and $B$ are parity check matrices of
Hamming $[n_a,k_a,3]_q$ and  $[n_b,k_b,3]_q$ codes, $C_A$ and $C_B$, respectively,
where $n_a = (q^{m_a}-1)/(q-1) \geq 3$,~ $n_b=(q^{m_b}-1)/(q-1)
\geq 3$,~ $k_a=n_a-m_a$, $k_b=n_b-m_b$ and where
\[
n=n_an_b,\;\;k= n - m_am_b,\;\;d=3,\;\; \rho=\min\{m_a,m_b\}.
\]
Then the code $C$ is completely transitive and, therefore, completely
regular with covering radius $\rho=\min\{m_a,m_b\}$ and
intersection numbers
\begin{equation}\label{eq:2.100}
\begin{split}
b_\ell = \frac{(q^{m_a}-q^\ell)(q^{m_b}-q^\ell)}{(q-1)},\;\;\ell= 0,\dots, \rho-1,\;\;\\
c_\ell =  q^{\ell-1} \frac{q^\ell-1}{q-1},\;\;\ell= 1,\dots,
\rho.
\end{split}
\end{equation}
\end{theorem}

\begin{definition}
Let $C$ be the $[n,k,d]_q$ code with parity check matrix $H$
where $1\leq k \leq n-1$ and $d\geq 3$. Denote by $C_{r}$ the
$[n,k,d]_{q^r}$ code over $\F_{q^r}$ with the same parity check
matrix $H$. Say that code $C_r$ is obtained by lifting $C$ to $\F_{q^r}$.
\end{definition}

In \cite{Rif3} we proved the following result

\begin{theorem}\label{theo:2.2}
Let $C_r(H_m^q)$ be the $[n,n-m,3;\rho]_{q^r}$ code of length $n =
(q^m-1)/(q-1)$ over the field $\F_{q^r}$ obtained by lifting a
$q$-ary perfect $[n,n-m,3]_q$ code $C(H^q_m)$ with parity check
matrix $H_m^q$. Then, code $C_r(H_m^q)$ is completely regular
with covering radius $\rho=\min\{m,r\}$ and intersection numbers
given by (\ref{eq:2.100}) taking $m_a=m$ and $m_b=r$.
\end{theorem}

From the above theorems we are giving completely regular codes
with different parameters and over different alphabets but with
the same intersection arrays.
Now our purpose is to consider deeper this coincidence.

Let $H^q_m = [\bh^{(1)}|\cdots |\bh^{(n)}]$ be a parity check matrix for the $q$-ary perfect
$[n,n-m,3]_q$ code $C=C(H^q_m)$.
Any codeword $\bv$ of $C$ is defined by the following equation
\EQ\label{eq:2.2} v_1\,\bh^{(1)} + v_2\,\bh^{(2)} + \ldots + v_n\,\bh^{(n)} =
0,\;\;v_i\in \F_q. \EN

A codeword $\bv$ of the lifted code $C_r(H^q_m)$ is defined by the same equation
\EQ\label{eq:2.3} v_1\,\bh^{(1)} + v_2\,\bh^{(2)} + \ldots + v_n\,\bh^{(n)} =
0, \EN
with the only difference that the unknown elements $v_i$ belong to
$\F_{q^r}$. Since any element $v_i$ of $\F_{q^r}$ can be presented as a vector
$\bv_i = (v_{i,1}, v_{i,2}, \ldots, v_{i,r})$ of
length  $r$ over $\F_q$, the equation (\ref{eq:2.3}) is transformed to the
system of equations
\EQ\label{eq:2.4} v_{1,\ell}\,\bh^{(1)} + v_{2,\ell}\,\bh^{(2)} + \dots + v_{n,\ell}\,\bh^{(n)} =
0,\;\;\ell=1,\ldots,r. \EN
Since for any $\ell$,\, $\ell\in \{1,\ldots,r\}$, the solutions $(v_{1,\ell}, v_{2,\ell}, \ldots, v_{n,\ell})$
of (\ref{eq:2.4}) are also solutions of (\ref{eq:2.2}), we conclude that (\ref{eq:2.3})
has $\left(q^{n-m}\right)^r = q^{(n-m)r}$ solutions. Indeed, matrix $H^q_m$ has
$n-m$ linearly independent rows.

Now, consider a code $D$ obtained by the Kronecker product, i.e. $D$ has
a parity check matrix $H=A\otimes B$. So, column vectors of $H$ can be seen as $\ba^{(i)}\otimes \bb^{(j)}$. Assume that $A$ is of
size $m \times n$ and $B$ is of size $m_b \times n_b$, where both matrices
are over $\F_q$. So, a codeword  $\bv$ of $D$
\[
\bv = (\bv_1, \bv_2, \ldots, \bv_{n}),\;\;\bv_i=(v_{i,1}, \ldots, v_{i,n_b}),
\]
is defined by the equation
\EQ\label{eq:2.5} \sum_{i=1}^{n} \ba^{(i)}\sum_{j=1}^{n_b} v_{i,j}\,\bb^{(j)}=0\,,
\EN
or by the system of equations
\EQ\label{eq:2.6} \sum_{i=1}^{n} \ba^{(i)}\sum_{j=1}^{n_b} v_{i,j}b^{(j)}_s=0,\;\;
s=1,\ldots, m_b,
\EN
which can be rewritten as follows:
\EQ\label{eq:2.7}  w_{1,s}\,\ba^{(1)} + w_{2,s}\,\ba^{(2)}
+ \ldots + w_{n,s}\,\ba^{(n)}=0,\;\;
s=1,\ldots, m_b,
\EN
where
\EQ\label{eq:2.8}
w_{i,s} = \sum_{j=1}^{n_b} v_{i,j}b^{(j)}_s.
\EN
Now the following observation gives an explanation of the coincidence
of the intersection numbers of both constructions:
{\em when vector $(v_{i,1}, v_{i,2}, \ldots, v_{i,n_b})$ in (\ref{eq:2.8}) runs over all
possible values in $\F_q^{n_b}$, vector $(w_{i,1},w_{i,2}, \ldots, w_{i,m_b})$
runs over all possible values in $\F_q^{m_b}$.} It is so, because the matrix
$B$ is of size $m_b \times n_b$ and has rank $m_b$. Hence the linear system
(\ref{eq:2.8}) defines an homomorphism of vectors from $\F_q^{n_b}$ onto
$\F_q^{m_b}$, whose kernel has cardinality $q^{n_b-m_b}$.

Compare the equations (\ref{eq:2.4}) and (\ref{eq:2.7}). If $r=m_b$ they look
identically. So, for the graphs defined by the coset graphs we
have the same conditions. But the corresponding codes are different (since they
have different lengths). The difference is that in (\ref{eq:2.7}) the variables $w_{i,s}$
are not codewords of $D$, but define the codewords through the linear system
(\ref{eq:2.8}).  Hence, from the point of view of the shape of parity check
equations (in both cases, the columns, $\bh^{(i)}$ and $\ba^{(i)}$, are all
linearly independent vectors of
corresponding lengths), the Kronecker product construction can be considered as a special type of the lifting construction. This is an explanation
why the intersection arrays of different completely regular codes from different
constructions coincide (compare Theorems \ref{theo:2.1} and \ref{theo:2.2}).

\section{Extending the Kronecker product construction}

Recall that by $C(H)$ we denote the code defined by
the parity check matrix $H$, by $H^q_m$ we denote the parity check matrix
of the $q$-ary Hamming $[n,n-m,3]_q$ code $C=C(H^q_m)$ of length $n=(q^m-1)/(q-1)$,
and by $C_r(H^q_m)$ we denote the code (of the same length $n=(q^m-1)/(q-1)$)
obtained by lifting $C(H^q_m)$ to the field $\F_{q^r}$.

Considering the above Kronecker construction (Theorem \ref{theo:2.1}) we could see that the
alphabets of both matrices $A=[a_{i,j}]$ and $B$ should be
compatible to each other in the sense that the multiplication $a_{i,j}B$
can be carried out. To have this compatibility it is enough that, say, the matrix
$A$ is over $\F_{q^u}$ and $B$ is over $\F_q$. First, we consider
the covering radius of the resulting codes.

\begin{lemma}\label{lem:4.1}
Let $C(H^{q^u}_{m_a})$ and $C(H^{q}_{m_b})$ be two Hamming codes
with parameters $[n_a,n_a-m_a,3]_{q^u}$ and $[n_b,n_b-m_b,3]_q$,
respectively, where $n_a = (q^{u\,m_a}-1)/(q^u-1)$,~
$n_b=(q^{m_b}-1)/(q-1)$,~ $q$ is a prime power,~ $m_a, m_b
\geq 2$, and $u \geq 1$. Then the code $C$ with parity check
matrix $H = H^{q^u}_{m_a} \otimes H^q_{m_b}$, the Kronecker
product of $H^{q^u}_{m_a}$ and $H^q_{m_b}$, has covering
radius $\rho = \min\{u\,m_a,\,m_b\}$.
\end{lemma}
\begin{proof}
Assume that the matrices $H$, $H^{q^u}_{m_a}$, and $H^q_{m_b}$
have columns $\bh_i$, $\ba_j$, and $\bb_s$, respectively, i.e.
they look as
\[
H = [\bh_1 | \cdots | \bh_n],\;\;H^{q^u}_{m_a} = [\ba_1 | \cdots |
\ba_{n_a}],\;\;H^{q}_{m_b} = [\bb_1 | \cdots | \bb_{n_b}].
\]
We have to prove that any column vector $\bx \in (\F_{q^u})^{m_am_b}$
can be presented as a linear combination of not more than $\rho$
columns of $H$.

By construction the column $\bh_i$ looks as
\[
\bh_i^T = [a_{1,j}\bb_s, a_{2,j}\bb_s, \ldots, a_{m_a,j}\bb_s],
\]
where $i=1, \ldots, n$, $n=n_an_b$, $j=1, \ldots, n_a$ $s=1, \ldots, n_b$ and $\bh_i^T$ is the transposed vector of $\bh_i$.
By definition, the matrix $H^q_{m_b}$ contains as column vectors
any vector $\by\in (\F_q)^{m_b}$ over the ground field $\F_q$, up to multiplication by scalars of $\F_q$. But, vectors $\bx$ are arbitrary vectors
over the extended field $\F_{q^u}$. Any such vector can be presented as a linear
combination of $u$ or less vectors from $H^q_{m_b}$. Hence for any choice of
$\bx=(\bx_1,\ldots,\bx_{m_a})$
we can always take not more than $u\,m_a$ columns of $H$ to have
$m_a$ equalities of the type
\[
\bx_i = \sum_{s=1}^u \alpha_s \bb_{i_s},\;\;i=1, \ldots, m_a,\, \alpha_i \in \F_{q^u},
\]
implying that
\[
\rho \leq u \, m_a.
\]
From the other side, by permuting the rows of $H$, the column
$\bh_i$ can be presented (with other index, say, $i'$) as follows:
\[
{\bh^T_{i'}} = [b_{1,s}\ba_j, b_{2,s}\ba_j, \ldots,
b_{m_b,s}\ba_j].
\]
Since vector $\ba_j$ is over the extended field $\F_{q^u}$, we can choose as
$\by\in (\F_{q^u})^{m_a}$ (a component of $\bx$) a vector $\F_{q^u}$, which can be presented only as some
vectors $\ba_j$, up to scalar, giving that
\[
\rho \leq  m_b.
\]
Since in both cases the bounds can be reached by appropriate choices of vector $\bx$,
we obtain the result.\qed
\end{proof}
We give also several simple facts from \cite{Rif2,Rif3}, which
will be used in the proof of the forthcoming theorem. As we said
before (\ref{eq:2.0}), any codeword $\bc\in C$ can be seen as a ($n_b\times
n_a$)-matrix $\left[\bc\right]$.

For a codeword $[\bc]$ define its syndrome, which, in a
matrix representation, is a $(m_b\times m_a)$ matrix
$S_c = [(A\otimes B)\bc^T]$ which is equal to zero. We have
\[
S_c = [(A\otimes B)\bc^T] = B [\,\bc\,] A = 0.
\]

From the equality above we see that any ($n_b\times n_a$)-matrix having
codewords of $C(A)$ as rows belongs to the code $C$, and any ($n_b\times
n_a$)-matrix with codewords of $C(B)$ as columns also belongs to
the code $C$. And vice versa, all the codewords in $C$ can always
be seen as linear combinations of matrices of the both types
above.

Fix a $1-1$ mapping $\mu$ from $\F_{q^u}$ to $(\F_q)^u$ writing for
any element $a \in \F_{q^u}$ its $\F_q$-presentation $\mu(a)$:
\[
\mu(a)=[a_0,a_1, \ldots, a_{u-1}]\;\;\longleftrightarrow\;\;a=\sum_{i=0}^{u-1}a_i\mu_i,
\]
where $\mu_0, \mu_1, \ldots, \mu_{u-1}$ is a fixed basis in
$\F_{q^u}$ over $\F_q$. Finally, extend the map $\mu$ to vectors
$\bv = (v_1, \ldots, v_n) \in (\F_{q^u})^n$ by the obvious way:
$\mu(\bv) = [\mu(v_1)\,|\, \cdots \,|\,\mu(v_n)]$.

\begin{definition}\label{de:4.1}
Given a vector $\bv\in (\F_{q^u})^n$, with syndrome $S_v$, which
is a $(m_b \times m_a)$ matrix over $\F_{q^u}$
denote by $\mu_v$ the $(m_b \times (u m_a))$
matrix obtained from $S_v$ using the map $\mu$ in its rows.
\end{definition}

We have the following three simple facts, which we formulate
into the next lemmas.

\begin{lemma}\label{lem:4.2}
Let $\bx,\by\in (\F_{q^u})^n$ be two vectors with syndromes $S_x$ and
$S_y$ and corresponding matrices $\mu_x$ and $\mu_y$, respectively. The equality
\[
(\mu_x)^T K = \mu_y,
\]
for any non-singular $m_b \times m_b$ matrix over $\F_q$ implies
the equality
\[
(S_x)^T K = S_y
\]
and vice versa.
\end{lemma}

\begin{lemma}\label{lem:4.3}
Let $H$ be any $(m \times n)$ matrix over $\F_q$ and $C_q(H)$
(respectively, $C_{q^u}(H)$) be the code over $\F_q$ (respectively,
over $\F_{q^u}$) with parity check matrix $H$. Then:
\begin{enumerate}[i)]
\item Any vector
$\bv = (v_1, \ldots, v_n) \in (\F_{q^u})^n$ is a codeword of $C_{q^u}(H)$
if and only if all rows of the matrix $[\mu(v_1)^T\,|\,\cdots \,|\,\mu(v_n)^T]$
are codewords of $C_q(H)$.
\item If $\varphi$ is an automorphism of $C_q(H)$, then it is an
automorphism of $C_{q^u}(H)$.
\end{enumerate}
\end{lemma}

The following statement generalizes the results of \cite{Rif2, Rif3}.

\begin{theorem}\label{theo:4.1}
Let $C(H^{q^u}_{m_a})$ and $C(H^q_{m_b})$ be two Hamming codes with parameters
$[n_a,n_a-m_a,3]_{q^u}$ and $[n_b,n_b-m_b,3]_q$, respectively, where
$n_a = (q^{u\,m_a}-1)/(q^u-1)$,~ $n_b=(q^{m_b}-1)/(q-1)$,~
$q$ is a prime power,~ $m_a, m_b \geq 2$,
and $u \geq 1$.
\begin{enumerate}[i)]
\item The code $C$ with parity check matrix $H = H^{q^u}_{m_a} \otimes
H^q_{m_b}$, the Kronecker product of $H^{q^u}_{m_a}$ and
$H^q_{m_b}$, is a completely transitive, and so completely regular, $[n, k, d;\rho]_{q^u}$ code
with parameters
\EQ\label{eq:4.11}
n=n_a\,n_b,~~k=n-m_a\,m_b,~~d=3,~~\rho=\min\{u\,m_a,~m_b\}.
\EN
\item The code $C$ has the intersection numbers:
\[
b_\ell ~=~ \frac{(q^{u\,m_a}-q^{\ell})(q^{m_b}-q^{\ell})}{(q-1)},\;\;\;\ell = 0, 1, \ldots, \rho-1,
\]
and
\[
c_{\ell} ~=~ q^{\ell-1}\frac{q^\ell-1}{q-1},\;\;\;\ell = 1, 2, \ldots, \rho.
\]
\item The lifted code $C_{m_b}(H^{q}_{um_a})$ is a completely regular code
with the same intersection array as $C$.
\end{enumerate}
\end{theorem}

\begin{proof}
The proof is mostly based on the same arguments which we used in the two
previous papers \cite{Rif2, Rif3}. We shortly repeat
only the places which it differ from the quoted papers.

First, from Lemma \ref{lem:4.1}, we have that $\rho = \min\{m_b, u\,m_a\}$.

The next step is to prove that the code $C$ is completely
transitive. This part is coming from similar arguments, which
we have used in the proof of \cite[Theo. 1]{Rif2}. The only
difference is that we have to use here Lemma \ref{lem:4.2} in
order to guaranty the existence of the invertible $m_b\times m_b$
matrix $K$ over $\F_q$ such that the equality $S^T_x K = S^T_y$
holds where $S_x$ and $S_y$ are the syndromes of vectors $\bx$ and
$\by$ over $\F_{q^u}$.

Denote by $C_A$ and $C_B$ the codes over $\F_{q^u}$, with parity
check matrices $A$ and $B$, respectively.

To prove that $C$ is a completely transitive code it is enough to
show that starting from two vectors $\bx, \by \in C(\ell)$, ~$1
\leq \ell \leq \rho$, there exists a monomial matrix $\varphi\in
\Aut(C)$ such that $\bx\varphi \in \by + C$ or, computing the
syndrome (Lemma \ref{lem:4.2}),
\[
S_{\bx\varphi}=[(A\otimes B)(\bx\varphi)^T] = [(A\otimes
B)(\by)^T].
\]
Let $\phi_1$ be any monomial $(n_a\times n_a$) matrix and $\phi_2$
be any monomial ($n_b\times n_b$) matrix. It is well known
\cite{Marc} that
$$(A\phi_1)\otimes (B\phi_2) = (A\otimes B)(\phi_1 \otimes \phi_2)$$
and $\phi_1\otimes \phi_2$ is a monomial ($n_an_b\times n_an_b$)
matrix.

Note that if $\varphi \in \Aut(C)$ then $H\varphi^T$ is a parity
check matrix for $C$ when $H$ is. Therefore, taking the specific
case where $\phi^T_1 \in \Aut(C_A)$ and $\phi^T_2 \in \Aut(C_B)$
we conclude that $(\phi^T_1 \otimes \phi^T_2)^T \in \Aut(C)$, or
the same $\phi_1 \otimes \phi_2 \in \Aut(C)$.

The two given vectors $\bx, \by$ belong to $C(\ell)$ and so, from~\cite{Rif2},
$\rank(S_x) = \rank(S_y)=\ell$, where $S_x$ and $S_y$ are the
syndrome of $\bx$ and $\by$, respectively. To prove that $C$ is a
completely transitive code we  show that there exists a
monomial matrix $\phi^T \in \Aut(C_B)$ such that
\begin{eqnarray*}
(A\otimes B)\by^T & = & (A \otimes B \phi)\bx^T\\
& = & (A \otimes B) ((I_{n_a}\otimes \phi)\bx^T)
\end{eqnarray*}
where $I_{n_a}$ is the $n_a \times n_a$ identity matrix.

Since $\ell \leq \rho \leq m_b$, it is straightforward to find an
invertible $(m_b\times m_b)$ matrix $K$ over $\F_q$ such that
$\mu^T_x\,K = \mu^T_y$. By Lemma \ref{lem:4.2} we conclude that
$S^T_x\,K=S^T_y$. Since $B$ is the parity check matrix of a
Hamming code, the matrix $K^T\,B$ is again a parity check matrix
for a Hamming code and $K^T B=B \phi$ for some monomial matrix
$\phi$. Moreover, if $G_B$ is the corresponding generator matrix
for this Hamming code, i.e. $B\,G_B^T=0$, then
$(B\phi)G_B^T= (K^T B)G_B^T=0$ and so $\phi^T \in \Aut(C_B)$.

Finally, we have
\begin{eqnarray*}
(A\otimes B)\by^T& = &S_y = K^T\,S_x = K^T (B[\bx] A^T)\\
& = & B\phi [\bx] A^T= (A\otimes B\phi)\bx^T\\
& = &(A\otimes B)((I_{n_a}\otimes \phi)\bx^t).
\end{eqnarray*}

Since the code $C$ is completely
transitive we conclude that $C$ is completely regular with the parameters (\ref{eq:4.11}). This gives item i).
Now we have to write down the expressions for all intersection numbers. In this case we use the same approach as in \cite{Rif3}.

We begin computing $b_0$, so the number of vectors in $C(1)$
which are at distance one from one given vector in $C$. Without
loos of generality (since $C$ is a linear code we can fix
the zero codeword $\bo$ in $C$ and count how many different vectors
of weight one there are in $C(1)$. The answer is immediately
\[
b_0~=~n\,(q^u-1) ~=~ \frac{(q^{u\,m_a} - 1)(q^{m_b}-1)}{(q - 1)}.
\]
Since the code $C$ has minimum distance $d=3$, we have $c_1=1$.

In general, let $1 \leq i \leq \rho-1$.  Take a  $(u\,m_a \times m_b)$-matrix
$E$ of rank $i$, over $F_q$, and compute the value $b_i$ as the number of
different $(u\,m_a \times m_b)$-matrices ${\bar E}$, over $F_q$, of
rank $i + 1 \leq \rho$, such that $E-{\bar E}$ has only one nonzero row. This
value is well known (see reference in \cite{Rif3}):
$$
b_i =  \frac{(q^{u\,m_a}-q^{i})(q^{m_b}-q^{i})}{(q-1)}\,.
$$
Now, using the expressions for $b_{i-1}$, $\mu_i$ and $\mu_{i-1}$ from Lemma
\ref{lem:2.2}, we obtain
\[
c_i = \frac{\mu_{i-1}b_{i-1}}{\mu_i}= q^{i-1}\frac{(q^{i}-1)}{q-1},
\]
i.e., we have item $ii)$.

The last statement $iii)$ follows directly from Theorem \ref{theo:2.2}.\qed
\end{proof}
\begin{remark}
We have to remark here that in the statement $iii)$ we can not
choose the code $C_{m_b}(H^{q^u}_{m_a})$ (instead of
$C_{m_b}(H^{q}_{um_a})$), which seems to be natural. We
emphasize that the codes $C_{m_b}(H^{q}_{um_a})$ and
$C_{m_b}(H^{q^u}_{m_a})$ are not only different completely regular
codes, but they induce different distance-regular graphs with
different intersection arrays. So, the code
$C_{m_b}(H^{q}_{um_a})$ suits to the codes from $i)$ in the sense
that it has the same intersection array. For example, the code
$C_2(H_3^{2^2})$ induces a distance-regular graph with intersection array
$(315,240;1,20)$ and the code $C_2(H_6^2)$ gives a distance-regular graph with intersection array $(189,124;1,6)$. To reach these results in both cases we use
the same Theorem \ref{theo:2.2}.
\end{remark}
\begin{remark}
The above theorem (Theorem \ref{theo:4.1}) can not be extended to the more general
case when the alphabets $\F_{q^a}$ and $\F_{q^b}$ of component
codes $C_A$ and $C_B$, respectively, neither $\F_{q^a}$ is a subfield of $\F_{q^b}$ or vice versa $\F_{q^b}$ is a subfield of $\F_{q^a}$. We illustrate it by
considering the smallest nontrivial example. Take two Hamming
codes, the $[5,3,3]$ code $C_A$ over $\F_{2^2}$ with parity check
matrix $H_2^{2^2}$, and the $[9,7,3]$ code $C_B$ over $\F_{2^3}$ with
parity check matrix $H_2^{2^3}$. Then the resulting $[45,41,3]$
code $C = C(H_2^{2^2} \otimes H_2^{2^3})$ over $\F_{2^6}$ is not
even uniformly packed in the wide sense, since it has the covering
radius $\rho=3$ and the outer distance $s=7$, which can be
checked by considering the parity check matrix of $C$.
\end{remark}

\section{Completely regular codes with different parameters, but the same
intersection array}

In \cite[Theo. 2.11]{Rif3} it is proved that lifting a $q$-ary Hamming code
$C(H_m^q)$ to $\F_q^s$  we obtain a completely regular code
$C_s(H_m^q)$ which is not necessarily isomorphic to the code $C_m(H_s^q)$.
However, both codes $C_s(H_m^q)$ and $C_m(H_s^q)$ have the same intersection array.
As we saw above, the code obtained by
the Kronecker product construction, or our extension for the case when the component codes have different alphabets, can have the same
intersection array. The next statements are the main results of our paper.

\begin{theorem} \label{theo:main}
Let $q$ be any prime number and let $a,b,u$ be any natural numbers. Then:
\begin{enumerate}[1)]
\item There exist the following
completely regular codes with different parameters $[n,k,d;\rho]_{q^r}$, where $d=3$ and $\rho=\min\{ua,b\}$:
\begin{enumerate}[i)]
\item $C_{ua}(H_b^{q})\;\mbox{over $\F_q^{ua}$ with}\;n = \frac{q^{b}-1}{q-1},\;k=n - b;$
\item $C_b(H_{ua}^{q})\;\mbox{over $\F_q^{b}$ with}\;n = \frac{q^{ua}-1}{q-1},\;k=n - ua;$
\item $C(H^q_b \otimes H^q_{ua}  )\;\mbox{over $\F_q$ with}\;n = \frac {q^{ua}-1}{q-1} \times \frac{q^b-1}{q-1} ,\;k=n - b ua;$
\item $C(H^q_b \otimes H^{q^a}_{u})\;\mbox{over $\F_q^{a}$ with}\;n = \frac{q^b-1}{q-1} \times \frac{q^{ua}-1}{q^u-1},\;k=n - bu;$
\item $C(H^q_b \otimes H^{q^u}_a)\;\mbox{over $\F_q^{u}$ with}\;n = \frac{q^b-1}{q-1} \times \frac{q^{ua}-1}{q^u-1},\;k=n - b a;$
\end{enumerate}
\item All the above codes have the same intersection numbers
\[
b_\ell=\frac{(q^b-q^{\ell})(q^{ua}-q^{\ell})}{(q-1)},\;\ell=0, \ldots, \rho-1,\;\; c_{\ell}=q^{\ell-1}\frac{q^\ell-1}{q-1},\;\ell=1, \ldots, \rho.
\]
\item All codes above coming from Kronecker constructions are completely transitive.
\end{enumerate}
\end{theorem}
\begin{proof}
The first two codes are obtained by the known lifting construction of the
corresponding perfect codes and they all come from Theorem \ref{theo:2.2}.
The third code is obtained by the known Kronecker product construction (both components
have the same alphabet) and come from Theorem \ref{theo:2.1}.
The two last codes are obtained by the Kronecker construction when the two component codes have different alphabets ($q$ and, $q^a$ or $q^u$))
and come from Theorem \ref{theo:4.1}.

For every code we find the covering radius and compute the intersection array using the corresponding
expressions given in the quoted theorems. All these codes have covering radius $\rho=\min\{ua,b\}$.

Complete transitivity of all codes coming from Kronecker constructions follows from Theorem \ref{theo:2.1} and Theorem \ref{theo:4.1}.\qed
\end{proof}

It is easy to see that the number of different completely transitive
(and, therefore, completely regular) codes with different parameters
and the same intersection array is growing. To be more specific we summarize
the results in the following Corollary, which comes
straightforwardly from the above Theorem \ref{theo:main}.

For a given natural number $n$, let $n_{(i)}$ be any divisor of
$n$ and $n^{(i)}=\frac{n}{n_{(i)}}$. Denote by $\tau(n)$ the number
of divisors of $n$.

\begin{corollary}
Given a prime power $q$ choose any two natural numbers $a,b >1$.
We can build the following completely regular codes with identical
intersection array and covering radius $\rho=\min\{a,b\}$. Specifically,
we construct a code over $\F_{q^r}$ for each divisor $r$ of $a$ or $b$
(the number of different codes is upper bounded by $\tau(a)+\tau(b)$) and we obtain:
\begin{enumerate}[i)]
\item Completely transitive codes $C(H^{q^{a^{(i)}}}_{a_{(i)}} \otimes H^{q}_{b})$ over $\F_{q^{a^{(i)}}}$, for any divisor $a_{(i)}$ of $a$, $a_{(i)}\neq 1$.
\item Completely transitive codes $C(H^{q}_{a} \otimes H^{q^{b^{(i)}}}_{b_{(i)}})$ over $\F_{q^{b^{(i)}}}$, for any divisor $b_{(i)}$ of $b$, $b_{(i)}\neq 1$.
\item Completely regular codes $C_{a}(H^{q}_{b})$ over $\F_{q^{a}}$ and $C_{b}(H^{q}_{a})$ over $\F_{q^{b}}$.
\end{enumerate}
\end{corollary}

\section{Uniformly packed codes}

Recall that a trivial $q$-ary repetition $[n, 1, n]_q$-code is a
perfect code if and only if $q=2$ and $n$ is odd. Denote by
$R^q_n=[I_{n-1}|-\bf{1}^T]$ the parity check matrix of the
$q$-ary repetition $[n,1,n]_q$ code. The following
statement generalizes the corresponding result of \cite{Rif2}.

\begin{theorem}\label{theo:4.3}
Let $C(H_m^{q^u})$ be the $q^u$-ary (perfect) Hamming
$[n,k,3]_{q^u}$-code of length $n_a=(q^{um}-1)/(q^u-1)$ and
$C(R^q_{n_b})$ be the  repetition $[n_b,1,n_b]_q$-code, where $q$
is a prime power,~$u \geq 1$,\, $m \geq 2$,\, $4 \leq n_b \leq
(q^u-1)n_a + 1$.
\begin{itemize}
\item The code $C = C(H_m^{q^u}\otimes R^q_{n_b})$ is a $q^u$-ary
uniformly packed (in the wide sense) $[n, k, d]_{q^u}$-code with
covering radius $\rho=n_b-1$ and parameters \EQ\label{eq:4.1} n =
n_a\,n_b,~~k = n - m\,(n_b-1),~~d = 3. \EN 
\item The code $C$ is
not completely regular.
\end{itemize}
\end{theorem}

\begin{proof}
First, we find the outer distance of code $C$. Using
\cite[Lemma 4]{Rif2} we see that any linear combination of rows of $H_m^{q^u}$
has weight either $q^{u(m-1)}$ or zero. By the same arguments used in
\cite[Theo. 3]{Rif2} (any row
of the parity check matrix of $C(R^q_{n_b})$ adds one more value
to the weight of $C^\perp$) we conclude that $s(C) = n_b-1$.

Now, about the covering radius of code $C$, we claim that
for the case $n_b \leq (q^u-1)n_a + 1$ we have $\rho = n_b-1$.

Since $\rho(C) \leq s(C)$ (Lemma \ref{lem:2.1}), it is enough to
show that $\rho(C) \geq n_b-1$. Take an arbitrary column vector
$\bx=(\bx_1,\bx_2,\ldots,\bx_{n_b-1})^T$ where $\bx_i$ is a vector
of length $m$ over $\F_{q^u}$. Present this vector
as a linear combination of columns of $H_m^{q^u}\otimes R^q_{n_b}$.
For any vector $\bx_i$ there is a column of $H_m^{q^u}$ which differs
from $\bx_i$ by a scalar. Choose as vectors $\bx_1,\bx_2,\ldots,\bx_{n_b-1}$
all possible different vectors of length $m$ over $\F_{q^u}$. This vector can be presented as a linear combination, at least, of
$n_b-1$ columns of $H_m^{q^u}\otimes R^q_{n_b}$. Hence,
$\rho \geq n_b-1$ and by Lemma \ref{lem:2.1} we conclude that $\rho = n_b-1$,
and, therefore, the resulting code is uniformly packed code.

Since $(q^u-1)n_a + 1 = q^{u m}$ for the case when $n_b \geq (q-1)n_a + 2$,
we can not choose all vectors $\bx_i$ such that they are different. So, if $n_b = (q-1)n_a + 2$,
for example, then two subvectors $\bx_i$ and, say, $\bx_j$ should be the same.
Now take as columns of $H_m^{q^u}\otimes R^q_{n_b}$ the column
$\bh = (\bh_1, \ldots, \bh_{n_b-1})^T$ with the same subcolumns $\bh_i=\bh_j$.
As a result we obtain $\rho=n_b-2$, but $s=n_b-1$, implying that the resulting
code is not uniformly packed.

To complete the proof we have to show that $C$ is not completely
regular. This comes by
the same argument used in \cite{Rif2}.
\end{proof}

\section{Coset distance-regular graphs}

Following \cite{Bro1}, we give some facts on
distance-regular graphs. Let $\Gamma$ be a finite connected simple
graph (i.e., undirected, without loops and multiple edges). Let
$d(\gamma, \delta)$ be the distance between two vertices $\gamma$
and $\delta$ (i.e., the number of edges in the minimal path
between $\gamma$ and $\delta$). The {\em diameter} $D$ of $\Gamma$
is its largest distance. Two vertices $\gamma$ and $\delta$ from
$\Gamma$ are {\em neighbors} if $d(\gamma, \delta) = 1$. Define
\begin{eqnarray*}
\Gamma_i(\gamma) ~=~\{\delta \in \Gamma:~d(\gamma, \delta) = i\}.
\end{eqnarray*}
An {\em automorphism} of a graph $\Gamma$ is a permutation $\pi$
of the vertex set of $\Gamma$ such that, for all $\gamma, \delta
\in \Gamma$ we have $d(\gamma,\delta)=1$ if and only if
$d(\pi\gamma,\pi\delta)=1$. Let $\Gamma_i$ be the graph with the
same vertices of $\Gamma,$ where an edge $(\gamma, \delta)$ is
defined when the vertices $\gamma, \delta$ are at distance $i$ in
$\Gamma$. Clearly, $\Gamma_1=\Gamma$. 
A
graph is called \textit{complete} (or a \textit{clique}) if any
two of its vertices are adjacent.
A connected graph $\Gamma$ with diameter $D\geq 3$ is called {\em
antipodal} if the graph $\Gamma_D$ is a disjoint union of
cliques~\cite{Bro1}. 

\begin{definition}\cite{Bro1}\label{14:de:1.3}
A simple connected graph $\Gamma$ is called {\em distance-regular}
if it is regular of valency $k,$ and if for any two vertices
$\gamma, \delta \in \Gamma$ at distance $i$ apart, there are
precisely $c_i$ neighbors of $\delta$ in $\Gamma_{i-1}(\gamma)$
and $b_i$ neighbors of $\delta$ in $\Gamma_{i+1}(\gamma)$.
Furthermore, this graph is called {\em distance-transitive}, if
for any pair of vertices $\gamma, \delta$ at distance $d(\gamma,
\delta)$ there is an automorphism $\pi$ from $\mbox{\rm
Aut}(\Gamma)$ which moves this pair $(\gamma,\delta)$ to any other
given pair $\gamma', \delta'$  of vertices at the same distance
$d(\gamma, \delta) = d(\gamma', \delta')$.
\end{definition}

The sequence $(b_0, b_1, \ldots, b_{D-1}; c_1, c_2, \ldots, c_D),$
where $D$ is the diameter of $\Gamma,$ is called the {\em
intersection array} of $\Gamma$. The numbers $c_i, b_i,$ and
$a_i,$ where $a_i=k- b_i - c_i,$ are called {\em intersection
numbers}. Clearly $b_0 = k,~~b_D = c_0 = 0,~~c_1 = 1.$

Let $C$ be a linear completely regular code with covering radius
$\rho$ and intersection array  $(b_0, \ldots , b_{\rho-1}; c_1,
\ldots c_{\rho})$. Let $\{B\}$ be the set of cosets of $C$. Define
the graph $\Gamma_C,$ which is called the {\em coset graph of
$C$}, taking all different cosets $B = C+ {\bf x}$ as vertices,
with two vertices $\gamma = \gamma(B)$ and $\gamma' = \gamma(B')$
adjacent if and only if the cosets $B$ and $B'$ contain neighbor
vectors, i.e., there are ${\bf v} \in B$ and ${\bf v}' \in B'$
such that $d({\bf v}, {\bf v}') = 1$.

\begin{lemma}\cite{Bro1,Ripu}\label{lem:2.5}
Let $C$ be a linear completely regular code with covering radius
$\rho$ and intersection array  $(b_0, \ldots , b_{\rho-1}; c_1,
\ldots c_{\rho})$ and let $\Gamma_C$ be the coset graph of $C$.
Then $\Gamma_C$ is distance-regular of diameter $D=\rho$ with the
same intersection array. If $C$ is completely transitive, then
$\Gamma_C$ is distance-transitive.
\end{lemma}

From all different completely transitive codes  described
above in Theorem \ref{theo:main}, we obtain distance-transitive graphs with
classical parameters (see \cite{Bro1}). These graphs have $q^{uab}$
vertices, diameter $D = \min\{ua,b\}$, and
intersection array given by
$$
b_l = \frac{(q^{ua}-q^{l})(q^{b}-q^{l})}{(q-1)}; \;
c_l=q^{l-1}\frac{q^l-1}{q-1},
$$
where $0 \leq l \leq D$.

Notice that bilinear forms graphs~\cite[Sec. 9.5]{Bro1} have the
same parameters and are distance-transitive too. These graphs are
uniquely defined by their parameters (see \cite[Sec. 9.5]{Bro1}). Therefore, all graphs coming from the completely regular and completely transitive codes described in Theorem \ref{theo:main} are bilinear forms graphs.
We did not find in the literature (in particular in
\cite{Del2}, where  the association schemes, formed by bilinear
forms, have been introduced, the description of these graphs,
as many different coset graphs of different completely regular codes. It is also known that these graphs are not antipodal
and do not have antipodal covers (see \cite[Sec. 9.5]{Bro1}). This can also be
easily seen from the proof of Lemma \ref{lem:4.1}. Indeed, a given vector
$\bx\in C(\rho)$ has many neighbors in $C(\rho)$.

\begin{theorem}\label{theo:graphs}
Let $C_1$, $C_2$, $\ldots$, $C_k$ be a family of linear completely transitive codes
constructed by
Theorem \ref{theo:4.1} and let
$\Gamma_{C_1}, \,\Gamma_{C_2},\, \ldots, \Gamma_{C_k}$ be their corresponding
coset graphs. Then:
\begin{enumerate}[i)]
\item Any graph $\Gamma_{C_i}$ is a distance-transitive graph, induced by
bilinear forms.
\item If any two codes $C_i$ and $C_j$ have the same intersection array, then
the graphs $\Gamma_{C_i}$ and $\Gamma_{C_j}$ are isomorphic.
\item If the graph $\Gamma_{C_i}$ has $q^{m}$ vertices, where $m$ is not a prime,
then it can be presented as a coset graph by several different ways,
depending on the number of factors of $m$.
\end{enumerate}
\end{theorem}

\begin{proof}
The proofs are straightforward. Given a completely transitive code $C_i$,
constructed by Theorem \ref{theo:4.1}, we conclude that the corresponding
coset graph is distance-transitive with the same intersection array (Lemma \ref{lem:2.5}).
Then, by using \cite[Sec. 9.5]{Bro1}, we conclude that this graph is
uniquely defined by their parameters and, therefore, it is induced by 
bilinear forms. Since two codes $C_i$ and $C_j$ with the same intersection arrays
induce two coset graphs with the same parameters, we conclude that
the corresponding graphs  $\Gamma_{C_i}$ and $\Gamma_{C_j}$ are isomorphic.
The last statement follows from Theorem \ref{theo:main}, since it gives
codes with the same intersection array.\qed
\end{proof}

\section{Conclusions}

In the current paper we use the Kronecker product
construction \cite{Rif2} for the case when component codes have
different alphabets and connect the resulting completely regular
codes with codes obtained by lifting  $q$-ary perfect codes.
This gives several different infinite classes of completely
regular codes with different parameters and with
identical intersection arrays. Given a prime power $q$ and any
two natural numbers $a,b$, we construct completely transitive
codes over different fields with covering radius $\rho=\min\{a,b\}$
and  identical intersection array, specifically, one code over
$\F_{q^r}$ for each divisor $r$ of $a$ or $b$.
We prove that the corresponding induced distance-regular coset
graphs are equivalent. In other words, the large class of
distance-regular graphs, induced by bilinear forms \cite{Del2},
can be obtained as coset graphs from different non-isomorphic
completely regular codes (either obtained by the
Kronecker product construction from perfect codes over different
alphabets, or obtained by lifting perfect codes \cite{Rif3}). Similar
results are obtained for uniformly packed codes in the wide sense.
Under the same conditions, explicit construction of an infinite
family of $q$-ary uniformly packed codes (in the wide sense) with
covering radius $\rho$, which are not completely regular, is also
given.

Finally, an open question arises: \textit{are bilinear forms graphs
the only distance-transitive graphs which have such many different presentations as coset graphs?}

\section{Acknowledgments}

This work has been partially supported by the Spanish MICINN grant
TIN2013-40524, the Catalan AGAUR grant 2014SGR-691, and also by
Russian fund of fundamental researches (15-01-08051).

\nocite{}

\end{document}